\documentclass{birkjour}
\usepackage{subfigure}
\usepackage{color}
\usepackage{verbatim}
 \newtheorem{thm}{Theorem}[section]
 \newtheorem{corollary}[thm]{Corollary}
 \newtheorem{lemma}[thm]{Lemma}
 \newtheorem{Proposition}[thm]{Proposition}
 \theoremstyle{definition}
 \newtheorem{defn}[thm]{Definition}
 \theoremstyle{remark}
 \newtheorem{remark}[thm]{Remark}
 \newtheorem{example}{Example}
 \numberwithin{equation}{section}

 \newcommand{\R}{\mathbb{R}}

\renewcommand{\L}{\mathcal{L}}

\begin{document}

%
%

\title[Dual relations between line congruences in $\mathbb{R}^3$ and surfaces in $\mathbb{R}^4$]
 { {\LARGE Dual relations between line congruences in $\mathbb{R}^3$ and surfaces in $\mathbb{R}^4$ }}

\author[M.Craizer]{Marcos Craizer}

\address{%
Departamento de Matem\'{a}tica- PUC-Rio\br
Rio de Janeiro, RJ, Brasil}
\email{craizer@puc-rio.br}

\author[R.A.Garcia]{Ronaldo Garcia}

\address{%
Instituto de Matem\'atica e Estat\'istica- UFG\br
Goi\^ania, GO, Brasil}
\email{ragarcia@ufg.br}

\thanks{ The authors want to thank CNPq,  PRONEX/ CNPq/ FAPEG 2017 10 26 7000 508, and CAPES (Finance Code 001)  for financial support during the preparation of this manuscript. \newline E-mail of the corresponding author: craizer@puc-rio.br}



\date{Januray 18, 2023}

\begin{abstract}
There is a natural duality between line congruences in $\mathbb{R}^3$ and surfaces in $\mathbb{R}^4$ that sends principal lines into asymptotic lines. The same correspondence takes the discriminant curve of a line congruence into the parabolic curve of the dual surface. Moreover, it takes the ridge curves to the flat ridge curves, while the subparabolic curves of a line congruence are taken to certain curves on the surface that we call flat subparabolic curves. In this paper, we discuss these relations and describe the generic behavior of the subparabolic curves at the discriminant curve of the line congruence, or equivalently, the parabolic curve of the dual surface. We also discuss Loewner's conjectures under the duality viewpoint.
\end{abstract}

\maketitle

\section{Introduction}

A line congruence is a smooth $2$-dimensional collection of lines in $\mathbb{R}^3$. It is a classical topic of research in geometry, and recently  some results concerning stable singularities have also been obtained (\cite{Craizer-Garcia},\cite{IST},\cite{Kabata},\cite{Lopes}).

We shall denote by $f:U\subset\mathbb{R}^2\to\mathbb{R}^3$ basepoints of the lines and by $\xi:U\to\mathbb{R}^3$ vectors in the direction of the lines, where $U$ is an open set of $\mathbb{R}^2$. The line congruence
\begin{equation}
F(u,v,t)=f(u,v)+t\xi(u,v), \ \ (u,v)\in U,\ \ \ t\in\mathbb{R},
\end{equation}
will be denoted $(f,\xi)$. It is proved in \cite{Craizer-Garcia} that generically a line congruence can be locally represented by a 
smooth immersion $f$ together with an equiaffine transversal vector field $\xi$, and in this paper we shall consider only such congruences.

Any curve in $U$ determines a ruled surface, and such a surface is developable when its tangent vector at each point belongs to at most two directions, called {\it principal (torsal) directions}. The line where both principal directions coincide is called the {\it discriminant curve}. The complement of the discriminant curve determine two open subsets of $U$, one where there are two distinct real principal directions, denoted $\bar{U}$, and the other where both principal directions are complex. 
The integral curves of the principal directions are called {\it principal (torsal) lines}.
The envelope $E$ of the lines are called the {\it focal set}, which is naturally divided into two parts $E_1$ and $E_2$ which meet at the image of the discriminant curve. At each part of the focal set, the singular and the parabolic curves are significant. Their pre-images are called {\it ridge} and {\it subparabolic} curves, respectively.

Consider an immersion $G:U\subset\mathbb{R}^2\to\mathbb{R}^4$. For a co-vector field $\eta$ along $G$, we can define the 
$\eta$-height function $g=g_{\eta}$ at a point $(u_0,v_0)\in U$ by
\begin{equation}\label{eq:HeightG}
g(u,v)=\eta(u_0,v_0)\cdot (G(u,v)-G(u_0,v_0)).
\end{equation}
It is not difficult to show that $(u_0,v_0)$ is a critical point of $g$ if and only if $\eta(u_0,v_0)$ vanishes on the tangent plane of $G$. Moreover, $(u_0,v_0)$ is a degenerate critical point of $G$ if $\eta(u_0,v_0)$ belongs at most two directions, called  {\it binormal directions}. For a binormal direction, a vector in the kernel of the hessian of $g$ is called an {\it asymptotic direction}. The integral curves of the asymptotic directions are called {\it asymptotic lines}. A point is called {\it $i$-flat ridge}, $i=1,2$, if, at this point, the $i$-asymptotic line has a higher order contact with the $i$-binormal hyperplane. We shall call a point {\it $i$-flat subparabolic} if, at this point, the $\bar{i}$-asymptotic line has a higher order contact with the $i$-binormal hyperplane, where
we are denoting $\bar{1}=2$ and $\bar{2}=1$.


Returning to line congruences, denote by $\nu$ is the co-normal vector field of the pair $(f,\xi)$. 
The support function $\rho$ with respect to the origin is defined as 
\begin{equation}\label{eq:Support}
\rho(u,v)=-\nu(u,v)\cdot f(u,v).
\end{equation}
The immersion $G:U\to\mathbb{R}^4$ given by
\begin{equation}\label{eq:Pedal}
G(u,v)=\left( \nu(u,v), \rho(u,v) \right),
\end{equation}
is called the projective pedal of $(f,\xi)$. Although $G$ depends on the equiaffine pair $(f,\xi)$ representing the line congruence, its asymptotic, flat ridges and flat subparabolic lines do not. Moreover, 
any immersion $G$ is the projective pedal of some equiaffine pair $(f,\xi)$.  We shall show that principal lines and the discriminant curve of a line congruence correspond to asymptotic lines and discriminant curve of its projective pedal. Moreover, ridge and subparabolic curves of the line congruence correspond to flat ridge curves and flat subparabolic curves of its projective pedal.

The generic behavior of ridge curves at the discriminant curve of a line congruence was studied in \cite{Craizer-Garcia}. In the present paper, we shall discuss the generic behavior of the subparabolic curve at the discriminant curve. We show that the focal set is hyperbolic at regular points of the discriminant
and that the subparabolic curve touches the discriminant curve only at its singular points, where they have a quadratic tangency. 

In last section of this paper, we shall discuss the implications of the above duality for the Loewner's conjecture concerning indices of umbilical points of a line congruence in $\mathbb{R}^3$ and inflection points of a surface in $\mathbb{R}^4$.

\section{Line Congruences in $\mathbb{R}^3$}

\subsection{Some basic concepts}

Along this paper, we shall consider line congruences $(f,\xi)$, where $f:U\subset\mathbb{R}^2\to\mathbb{R}^3$ is an immersion and $\xi:U\to\mathbb{R}^3$ is a transversal vector field. For $X,Y\in\mathfrak{X}(U)$, we write 
\begin{equation}\label{eq:DecompositionR3}
D_Xf_*Y=f_*(\nabla_XY)+h(X,Y)\xi, 
\end{equation}
where $\nabla$ is a connection and $h$ a bilinear form. The non-degeneracy of $h$ is independent of the choice of $\xi$, and in this case we say that the immersion $f$ is non-degenerate. 
For $X\in\mathfrak{X}(U)$,  write 
\begin{equation*}\label{eq:ShapeR3}
D_X\xi=-f_*(SX)+\tau(X)\xi,
\end{equation*}
where $S$ is a linear map called {\it shape operator} and $\tau$ a $1$-form. We say that $(f,\xi)$ is equiaffine if $\tau(X)=0$, for any $X\in\mathfrak{X}(U)$. It is proved in \cite{Craizer-Garcia} that generically a line congruence can be locally represented by a smooth immersion $f$ together with an equiaffine transversal vector field $\xi$, and in this paper we shall consider only such congruences.


The eigenvectors of the shape operator $S$ determine ai most two directions which coincide with the principal directions of the line congruence. For simplicity, we shall assume that the eigenvalues $\lambda_i$ of $S$, $i=1,2$, are not zero.
Then the $i$-focal surface is given by
\begin{equation*}
E_i=f+\frac{1}{\lambda_i}\xi,
\end{equation*}
and the focal surface $E$ is the union of $E_1$ and $E_2$. Note that $E_1$ and $E_2$ meet at the image of the discriminant curve. Sometimes we shall call the image of discriminant curve by $E_1$ (or $E_2$) simply by discriminant curve, hoping that this will cause no confusion. The principal directions, discriminant curves and focal surfaces depend only on the line congruence and not on the choice of the base surface $f$ or the director vector $\xi$.

\subsection{Ridge curves}

Consider a line congruence defined by the pair $(f,\xi)$.

\begin{defn}
The $i$-ridge curve is the set of points $(u,v)\in \bar{U}$ where the focal set $E_i$ is singular. 
\end{defn}

Assuming $(f,\xi)$ equiaffine, we have the following characterization of the ridge curves (\cite{Craizer-Garcia}):

\begin{Proposition}
Assume $(f,\xi)$ equiaffine. Then $(u,v)\in\bar{U}$ is a ridge point if and only if
$(\lambda_i)_{w_i}(u,v)=0$, where $w_i$ is the $i$-principal direction. 
\end{Proposition}

\begin{remark}\label{rem:Singularities}
Generically, the ridge curve is formed mostly by points where the principal direction is not tangent to it and some isolated points where the principal direction is tangent to it. At the former, the focal set is a cuspidal edge, while at the latter, the focal set is a swallowtail (\cite{Craizer-Garcia}).
\end{remark}

\subsection{Subparabolic curves}

The subparabolic curve is the pre-image of the parabolic set of the focal surface. This set is also called flexcord (\cite[ch.17]{Porteous}). For a discussion of alternative names for this curve, see \cite{Morris}. The definition of the subparabolic set was given for immersions $f$ together with an Euclidean normal unitary vector field $\xi$, and we propose below an extension for a general line congruence.

\begin{defn}
We say that $(u,v)\in\bar{U}$ is a {\it i-subparabolic} point of the line congruence if its image is a parabolic point of the focal set $E_i$.
\end{defn}

Consider a point which is neither a discriminant nor a ridge point. In a neighborhood of such a point, we may assume 
that $(u,v)$ are principal coordinates. More precisely, we shall assume along this section that $v=constant$ are the $1$-principal lines and $u=constant$ are the $2$-principal lines. Denote by $[v_1,v_2,v_3]$ the determinant of the matrix whose columns are $v_i$, $i=1,2,3$.

\begin{Proposition}\label{Prop:Subparabolic}
Assume that $(f,\xi)$ is equiaffine. Then $(u,v)$ belongs to the $1$-subparabolic set if and only if 
\begin{equation}
\left[ f_v, f_{vv}, \xi \right]=0.
\end{equation}
Similarly, $(u,v)$ belongs to the $2$-subparabolic set if and only if 
\begin{equation}
\left[ f_u, f_{uu}, \xi \right]=0.
\end{equation}
\end{Proposition}

\begin{proof}
We shall prove the first assertion, the second being similar. Since $E_1=f+\lambda_1^{-1}\xi$, we have that 
$$
(E_1)_u=-\frac{(\lambda_1)_u}{\lambda_1^{2}}\xi,\ \ (E_1)_v=\left(1-\frac{\lambda_2}{\lambda_1}\right)f_v-\frac{(\lambda_1)_v}{\lambda_1^{2}}\xi, 
$$
Thus
$$
\left[  (E_1)_u, (E_1)_v, (E_1)_{uu}\right]=\frac{((\lambda_1)_u)^2}{\lambda_1^{4}}\left(\lambda_1-\lambda_2\right)
[f_u, f_v, \xi],
$$

$$
\left[  (E_1)_u, (E_1)_v, (E_1)_{uv}\right]=0, 
$$

$$
\left[  (E_1)_u, (E_1)_v, (E_1)_{vv}\right]=-\frac{(\lambda_1)_u}{\lambda_1^{4}}\left(\lambda_1-\lambda_2\right)^2[f_v, f_{vv}, \xi],
$$
thus proving the proposition.
\end{proof}

\begin{remark}\label{rem:SubParabolicNonEuclidean}
For an immersion $f$ together with an Euclidean unitary normal vector field $\xi$, the condition for $1$-subparabolic point is equivalent to $(\lambda_1)_v=0$ and the condition for $2$-subparabolic point is $(\lambda_2)_u=0$ (\cite[th. 6.17]{IFRT},\cite[th.17.2]{Porteous}). But these equivalences do not hold for general equiaffine vector fields (for a counter-example, see Section \ref{sec:ContactSubParabolic}).
\end{remark}

\section{Surfaces in $\mathbb{R}^4$}\label{sec:Surfaces}



Consider an immersion $G:U\subset\mathbb{R}^2\to\mathbb{R}^4$ and denote by $N^*{G}$ the space of normal co-vector fields along $G$, i.e., co-vector fields that vanishes on the tangent plane of $G$. Let $\{\eta_1,\eta_2\}$ be linearly independent co-vector fields in $N^*G$. Then any $\eta\in N^*G$ can be written as $\eta=s_1\eta_1+s_2\eta_2$, where $s_1$ and $s_2$ are real functions on $U$. Let
$$
A_i=\eta_i\cdot G_{uu}(u_0,v_0),\ \ B_i=\eta_i\cdot G_{uv}(u_0,v_0),\ \ C_i=\eta_i\cdot G_{vv}(u_0,v_0),
$$
be the coefficients of the second fundamental form associated with $\eta_i$, $i=1,2$. Then $(u_0,v_0)$ is a degenerate critical point of $g=g_{\eta}$ given by Equation \eqref{eq:HeightG} if and only if
\begin{equation}\label{eq:Binormal}
(s_1A_1+s_2A_2)(s_1C_1+s_2C_2)-(s_1B_1+s_2B_2)^2=0.
\end{equation}
A point $(u_0,v_0)$ is called {\it parabolic}, {\it hyperbolic} or {\it elliptic} if the discriminant of Equation \eqref{eq:Binormal} is zero, positive or negative, repectively. 
For a hyperbolic point $(u_0,v_0)$, the solutions $(s_1,s_2)$ of Equation \eqref{eq:Binormal} define the {\it bi-normal directions} and the corresponding kernel of the hessian of $g$ defines the
{\it asymptotic directions}.  The integral curves of the asymptotic directions are the {\it asymptotic lines} (\cite{Bruce-Tari}, \cite[sec. 7.3]{IFRT}).

At a hyperbolic point of $G$, we may assume that $(u,v)$ are asymptotic coordinates and we shall do it along this section. Denote by $\eta_i$ the binormal co-vector field associated with $i$-direction, $i=1,2$. In other words, 
\begin{equation}\label{eq:BinormalG}
\eta_1\cdot G_{uu}=\eta_1\cdot G_{uv}=0, \ \ \eta_2\cdot G_{vv}=\eta_2\cdot G_{uv}=0.
\end{equation}
Observe that, for $\eta\in N^*G$, Equations \eqref{eq:BinormalG} are equivalent to
\begin{equation}\label{eq:BinormalG1}
(\eta_1)_{u}\cdot G_{u}=(\eta_1)_{u}\cdot G_{v}=0, \ \ (\eta_2)_v\cdot G_{v}=(\eta_2)_v\cdot G_{u}=0.
\end{equation}

\begin{lemma}\label{lemma:BinormalEta}
Let $\eta\in N^*G$. Then $\eta$ is the $1$-binormal if and only if $\eta_u\in N^*G$. Similarly, $\eta$ is the $2$-binormal if and only if $\eta_v\in N^*G$.
\end{lemma}
\begin{proof}
Straightforward from Equations \eqref{eq:BinormalG1}.
\end{proof}

\begin{defn}
We say that $(u_0,v_0)\in U$ is a {\it 1-flat ridge point} if $(\eta_1)_u$ is a multiple of $\eta_1$. Similarly, $(u_0,v_0)\in U$ is a {\it 2-flat ridge point} if $(\eta_2)_v$ is a multiple of $\eta_2$  (\cite[def. 7.6]{IFRT}).
\end{defn}

We next consider a new concept, the flat subparabolic point: 

\begin{defn}
We say that $(u_0,v_0)\in U$ is a {\it 1-flat subparabolic point} if $(\eta_1)_v$ is a multiple of $\eta_1$. Similarly, $(u_0,v_0)\in U$ is a {\it 2-flat subparabolic point} if $(\eta_2)_u$ is a multiple of $\eta_2$.
\end{defn}

It may occur that two different immersions have parallel tangent planes for any value of the parameter, as the following example shows us:

\begin{example}
Consider $G_1(u)=(\cos(u),\sin(u))$ and 
$$
\bar{G}_1(u)=\frac{1}{1+u}(\cos(u),\sin(u))-\frac{1}{(1+u)^2}(-\sin(u),\cos(u)).
$$
Then the tangent line of $G_1$ and $\bar{G}_1$ at the same parameter $u$ are parallel, for any $u$. Now if we consider
$G_2(v)=\bar{G}_2(v)=(0,v)$, the Cartesian product $G_1\times G_2$ and $\bar{G}_1\times\bar{G}_2$ have parallel tangent planes at $(u,v)$, for any $(u,v)$. 
\end{example}

Next proposition shows that all concepts defined in this section are equivalent for two surfaces with the same tangent planes:

\begin{Proposition}\label{prop:TangentPlanes}
If $G$ and $\bar{G}$ are two immersions with the same tangent planes at each $(u,v)\in U$, then the asymptotic lines, ridge and subparabolic curves are the same.
\end{Proposition}
\begin{proof}
First note that $N^*G=N^*\bar{G}$. For $\eta\in N^*G$, the conditions $\eta_u\in N^*G$ or $\eta_v\in N^*G$  are independent of the surface $G$ or $\bar{G}$. By Lemma \ref{lemma:BinormalEta}, the bi-normals co-vector fields are independent of the surface $G$ or $\bar{G}$. Moreover, since $\eta_i$, $i=1,2$, are independent of the choice of $G$ or $\bar{G}$, so are the flat ridge and flat subparabolic points.
\end{proof}

\section{Duality between line congruences in $\mathbb{R}^3$ and surfaces in $\mathbb{R}^4$}

\subsection{Centroaffine duality between pairs in $\mathbb{R}^4$}

Consider a centroaffine immersion $F:U\subset\mathbb{R}^2\to\mathbb{R}^4$ together with a transversal vector field $\Phi$.
This means that at each $(u,v)\in U$, $TF\oplus F\oplus \Phi=\mathbb{R}^4$. 
 Let $D$ be the canonical flat affine connection of $\R^{4}$. For $X,Y\in\mathfrak{X}(U)$, write
\begin{equation}\label{eq:Decomposition}
D_XF_*Y=T(X,Y)F+F_*(\nabla_XY)+H(X,Y)\Phi,
\end{equation}
where $H$ and $T$ are bilinear forms and $\nabla$ a torsion-free connection on $U$. The bilinear form $H$ is called the affine metric with respect to $\Phi$.
The conformal class of $H$ does not depend on $\Phi$, and thus non-degeneracy and also positiveness of $H$ are independent of the choice of $\Phi$. 
We shall assume non-degeneracy of $H$ throughout the paper.
We also write
\begin{equation}\label{eq:ShapeOperator}
D_X\Phi=\rho(X)F-F_*(SX)+\tau(X)\Phi,
\end{equation}
where $\rho$ and $\tau$ are $1$-forms and $S$ a $(1,1)$-tensor on $M$.
We say that the pair $(F,\Phi)$ is {\it equiaffine} if $\tau=0$.  


The centroaffine dual of the pair $(F,\Phi)$ is the pair
$(G,\Psi):U\to\mathbb{R}^{4}$ defined by the following equations:
\begin{equation}\label{eq:DefineG}
G\cdot\Phi=1,\ \ G\cdot F_*X=0,\ \ G\cdot F=0,
\end{equation}
\begin{equation}\label{eq:DefinePsi}
\Psi\cdot\Phi=0,\ \ \Psi\cdot F_*X=0,\ \ \Psi\cdot F=1.
\end{equation}
(see \cite{Nomizu2}). 
When $(F,\Phi)$ is equiaffine, the centroaffine dual of $(G,\Psi)$ is $(F,\Psi)$, i.e.,
\begin{equation}\label{eq:DefineG}
F\cdot\Psi=1,\ \ F\cdot G_*X=0,\ \ F\cdot G=0,
\end{equation}
\begin{equation}\label{eq:DefinePsi}
\Phi\cdot\Psi=0,\ \ \Phi\cdot G_*X=0,\ \ \Phi\cdot G=1.
\end{equation}
Moreover, the dual pair $(G,\Psi)$ is also equiaffine.

The following proposition can be found in \cite[Lemmas 3.2 and 3.3]{Nomizu2}:

\begin{lemma}\label{lemma:EqualH}
Consider a centroaffine immersion $F$ together with a transversal vector field $\Phi$, and
let $G$ defined by Equation \eqref{eq:DefineG}. Then
the affine metric of $G$ coincides with the affine metric of $F$. 

\end{lemma}

\begin{proof}
Observe that
$$
D_XG_*Y\cdot F=-G_*Y\cdot F_*X=G\cdot D_YF_*X,
$$
thus proving the lemma.
\end{proof}

For more details on centroaffine duality, see \cite{Nomizu2}.

\subsection{The projective pedal of a line congruence}\label{sec:Pedal}

Consider a pair $(f,\xi)$.  
A lifting of $(f,\xi)$ is an immersion $F$ and a transversal vector field $\Phi$ of the form
\begin{equation}\label{eq:Lifting}
F=\lambda\left( f, 1\right),\ \ \Phi=\left( \xi+\mu f, \mu   \right),
\end{equation}
where $\lambda$ and $\mu$ are arbitrary functions of $(u,v)$.
Denoting by $(G,\Psi)$ the centroaffine dual of $(F,\Phi)$, one can verify that
the immersion $G$ does not depend on the lifting, which implies that $G$ is projectively invariant (\cite[Note 9]{Nomizu}). We call $G$ the {\it projective pedal} of the pair $(f,\xi)$.

Assume that $(f,\xi)$ is equiaffine. Then the lifting obatined by choosing $\lambda=1$ and $\mu=0$ in Equation \eqref{eq:Lifting},
\begin{equation}\label{eq:Lifting0}
F=\left( f, 1\right),\ \ \Phi=\left( \xi, 0   \right),
\end{equation}
is also equiaffine. From now on, we shall consider only the lifting given by Equation \eqref{eq:Lifting0}.

One can easily verify that the centroaffine dual $(G,\Psi)$ of $(F,\Phi)$ is exactly the projective pedal of the pair $(f,\xi)$ given by Equation \eqref{eq:Pedal}, where $\nu$ is the co-normal and $\rho$ the support function of $(f,\xi)$. More precisely,
\begin{equation}\label{CentroAffineDual}
G=\left(\nu, -\nu\cdot f\right),\ \ \ \Psi=(0,1).
\end{equation}
One can also verify that  $H=h$, where $h$ is given by Equation \eqref{eq:DecompositionR3}
and $H$ is given by Equation \eqref{eq:Decomposition}. By Lemma
\ref{lemma:EqualH}, this implies that $G$ is non-degenerate if and only if $(f,\xi)$ is non-degenerate.

If we change the equiaffine pair $(f,\xi)$ representing the line congruence, the projective pedal $G$ may change to $\bar{G}$. But it is clear that the cotangent bundle of $G$ and $\bar{G}$ coincide, since they are generated by $F$ and $\Phi$. Then, by Proposition \ref{prop:TangentPlanes}, the concepts considered in Section \ref{sec:Surfaces} remain invariant. 

We can also prove the following:

\begin{Proposition}
Given a non-degenerate surface $G$ in $\mathbb{R}^4$, there exists a non-degenerate equiaffine pair $(f,\xi)$ whose projective pedal is exactly $G$.
\end{Proposition}

\begin{proof}
Given a non-degenerate surface $G$ in $\mathbb{R}^4$, there exists a constant vector field which is transversal to $G$. By a linear transformation, we may assume that this vector is $\Psi=(0,1)$. Let $(F,\Phi)$ be the centroaffine dual of $(G,\Psi)$. Then, by the dual relations, we can write $F=(f,1)$ and $\Phi=(\xi,0)$. This implies that we can write $G$ in the form \eqref{CentroAffineDual}.
\end{proof}

\section{Properties of the duality}

In this section, we prove the dual relations between line congruences in $\mathbb{R}^3$ and surfaces in $\mathbb{R}^4$. The main tool is the existence, outside umbilical points, of a non-degenerate equiaffine pair $(f,\xi)$ representing the line congruence
(\cite{Craizer-Garcia}).

\subsection{Duality between principal and asymptotic lines}

Assume $(f,\xi)$ is an equiaffine pair. In a neighborhood of a point $(u_0,v_0)\in\bar{U}$, we may choose principal coordinates $(u,v)$, i.e., 
$$
\xi_u=-\lambda_1f_u, \ \ \xi_v=-\lambda_2 f_v. 
$$
Denote 
$$
h_{11}=h\left( \frac{\partial}{\partial u},  \frac{\partial}{\partial u} \right), \ h_{12}=h\left( \frac{\partial}{\partial u},  \frac{\partial}{\partial v} \right),\ h_{22}=h\left( \frac{\partial}{\partial v},  \frac{\partial}{\partial v} \right).
$$
Since the shape operator is $h$-self-adjoint, we have that, in principal coordinates, $h_{12}=0$. 

\begin{lemma}
Let $G$ be the projective pedal of $(f,\xi)$. 
We have that
$$
F\cdot G_{uu}=h_{11},\ \ F\cdot G_{uv}=0,\ \ F\cdot G_{vv}=h_{22},
$$
$$
\Phi\cdot G_{uu}=-\lambda_1h_{11},\ \ \Phi\cdot G_{uv}=0,\ \ \Phi\cdot G_{vv}=-\lambda_2h_{22}.
$$
\end{lemma}
\begin{proof}
We prove the lemma for $G_{uu}$, the case of $G_{uv}$ and $G_{vv}$ being analogous.
We have 
$$
G_{uu}=\left( \nu_{uu}, -\nu_{uu}\cdot f \right)+\left(0, -\nu_u\cdot f_u  \right).
$$
Thus $F\cdot G_{uu}=h_{11}$ and 
$$
\Phi\cdot G_{uu}=\nu_{uu}\cdot\xi=-\nu_u\cdot\xi_u=\lambda_1\nu_u\cdot f_u=-\lambda_1h_{11},
$$
thus proving the lemma. 
\end{proof}

\begin{Proposition}\label{Prop:Eta12}
The binormals are
$$
\eta_1=\Phi+\lambda_1F, \ \ \eta_2=\Phi+\lambda_2F.
$$
corresponding to the $u$ and $v$ directions, respectively. 
\end{Proposition}

\begin{proof}
We prove the proposition for $\eta_1$, the case of $\eta_2$ being analogous. Observe that
$$
\eta_1\cdot G_{uu}=0,\ \ \eta_1\cdot G_{uv}=0.
$$
This implies that the hessian of $g=g_{\eta_1}$ at $(u_0,v_0)$ is degenerate with $\frac{\partial}{\partial u}$ in its kernel.
\end{proof}

\begin{corollary}\label{cor:PrincipalAsymptotic}
Principal lines of a line congruence correspond to asymptotic lines of the projective pedal. Thus the discriminant curve of the line congruence corresponds to the parabolic curve of the projective pedal.  
\end{corollary}

\subsection{Duality between ridge and flat ridge curves}

Assume $(f,\xi)$ is an equiaffine pair and that $(u,v)$ are principal coordinates for the line congruence. Then, by Corollary \ref{cor:PrincipalAsymptotic}, $(u,v)$ are asymptotic coordinates for $G$.

\begin{Proposition} 
Ridge curves of a line congruence correspond to flat ridge curves of the projective pedal. 
\end{Proposition}

\begin{proof}
We prove the proposition for $i=1$, the case $i=2$ being similar. The
$1$-binormal co-vector field is given by
$$
\eta_1=\Phi+\lambda_1F=(\xi+\lambda_1f, \lambda_1).
$$
Thus 
$$
(\eta_1)_u=(\lambda_1)_u\left( f, 1  \right).
$$
We conclude that $(\eta_1)_u$ is a multiple of $\eta_1$ if and only if $(\lambda_1)_u=0$. 
\end{proof}


\subsection{Duality between subparabolic and flat subparabolic curves}

Assume $(u_0,v_0)$ is a non-ridge and non-discriminant point of a line congruence $(f,\xi)$. From \cite{Craizer-Garcia}, we may assume that $(f,\xi)$ is a non-degenerate equiaffine pair and that $(u,v)$ are principal coordinates.
Then $h_{12}=0$, which implies that $G_{uv}$ is tangent to $G$. Moreover, $h_{11}\neq 0$ and $h_{22}\neq 0$, which implies that $G_{uu}$ and $G_{vv}$ are not tangent to $G$.

Recall that
$$
G=\left( \nu, -\nu\cdot f \right)
$$
and 
$$
\eta_1=(\xi+\lambda_1f, \lambda_1).
$$

\begin{lemma}
We have that 
$$
\eta_1\cdot G_{uuv}=0, \ \ \ \eta_1\cdot G_{uvv}=(\lambda_2-\lambda_1)\nu_{uv}\cdot f_v.
$$
\end{lemma}

\begin{proof}
Since
$$
G_{uv}=\left( \nu_{uv}, -\nu_{uv}\cdot f \right), 
$$
we have that
$$
\eta_1\cdot G_{uuv}=\nu_{uuv}\cdot\xi-\lambda_1\nu_{uv}\cdot f_u=-\nu_{uv}\cdot\xi_u-\lambda_1\nu_{uv}\cdot f_u=0.
$$
Moreover
$$
\eta_1\cdot G_{uvv}=\nu_{uvv}\cdot\xi-\lambda_1\nu_{uv}\cdot f_v=-\nu_{uv}\cdot\xi_v-\lambda_1\nu_{uv}\cdot f_v=(\lambda_2-\lambda_1)\nu_{uv}\cdot f_v,
$$
thus proving the lemma.
\end{proof}

\begin{corollary}
Subparabolic points of the line congruence correspond to flat subparabolic points of the projective pedal.
\end{corollary}

\begin{proof}
Since $(f,\xi)$ is equiaffine, by Proposition \ref{Prop:Subparabolic}, a $1$-subparabolic point is characterized by the condition $\nu_u\cdot f_{vv}=0$, or equivalently, $\nu_{uv}\cdot f_v=0$.
From the above lemma, this condition is equivalent to $\eta_1\cdot G_{uvv}=0$. 

Differentiating the equations
$$
\eta_1\cdot G_u=0, \ \eta_1\cdot G_{uu}=0, \ \eta_1\cdot G_{uv}=0,
$$
we obtain
$$
(\eta_1)_v\cdot G_u=0, \ (\eta_1)_v\cdot G_{uu}=0, \ (\eta_1)_v\cdot G_{uv}=-\eta_1\cdot G_{uvv}.
$$
The above equations imply that $(\eta_1)_v$ is a multiple of $\eta_1$ if and only if $\eta_1\cdot G_{uvv}=0$, or equivalently, 
if $(u,v)$ is a subparabolic point of the line congruence. 
\end{proof}

\section{Subparabolics at the Discriminant Curve}

Given a line congruence, its discriminant is the set of points where the principal directions coincide. Generically, the discriminant set is a smooth curve. For most points, the principal direction is not tangent to the curve itself, and these points are called {\it regular discriminant points}. But there are some isolated points where the principal direction is tangent to the discriminant curve, and these points are called {\it singular discriminant points}. At regular discriminant points, the focal set $E=E_1\cup E_2$ is a smooth surface, while at a singular point the focal set is generically a cuspidal edge (\cite{Craizer-Garcia}). 

\subsection{Regular discriminant points}

\begin{Proposition}\label{prop:TangentPlaneFocal}
Consider a line congruence defined by the equiaffine pair $(f,\xi)$. At a regular point of the discriminant, the tangent plane of the focal set $E$ is generated by $\{f_*w,\xi\}$, where $w$ denotes the unique principal direction.
\end{Proposition}

\begin{proof}
Consider a regular point $(u_0,v_0)$ of the discriminant. At a point $(u_1,v_1)\in\bar{U}$ close to $(u_0,v_0)$, consider coordinates $(t,r)$ such that the line $f_t$ is $1$-principal direction and $f_r$ is close to the tangent line of the discriminant curve. Then
$$
(E_1)_t(u_1,v_1)=-\frac{(\lambda_1)_t}{\lambda_1^2}\xi.
$$
Thus $\xi$ is tangent to $E_1$ at $(u_1,v_1)$. Since $E$ is smooth at $(u_0,v_0)$, making $(u_1,v_1)\to (u_0,v_0)$ we conclude that $\xi$ is tangent to $E$ at $(u_0,v_0)$. Moreover, 
$$
(E_1)_r(u_1,v_1)=f_*\left( (I-\lambda_1^{-1}S)\frac{\partial}{\partial r} \right)-\frac{(\lambda_1)_r}{\lambda_1^2}\xi.
$$
which implies that 
$$
f_*\left( (I-\lambda_1^{-1}S)\frac{\partial}{\partial r} \right)
$$
is tangent to $E_1$ at $(u_1,v_1)$. Since $I-\lambda_1^{-1}S$ has a double $0$ eigenvalue at $(u_0,v_0)$, it is nilpotent of degree $2$, and so its image equals its kernel. Thus $(I-\lambda_1^{-1}S)\frac{\partial}{\partial r}$ is a multiple of $\frac{\partial}{\partial t}$ at $(u_0,v_0)$. Making $(u_1,v_1)\to(u_0,v_0)$, we conclude that $f_*w$ is tangent to $E$ at $(u_0,v_0)$.
\end{proof}

\begin{Proposition}
At regular discriminant points, the focal set is hyperbolic and $\xi$ is an asymptotic direction. 
\end{Proposition}

\begin{proof}
Assume that $(f,\xi)$ is equiaffine in a neighborhood of a regular point $(u_0,v_0)$ of the discriminant. Then
$$
D_w\xi=-\lambda f_*w
$$
where $w$ is an eigenvector and $\lambda=\lambda_1=\lambda_2$ is the eigenvalue of $S$ at $(u_0,v_0)$. By Proposition \ref{prop:TangentPlaneFocal}, $D_w\xi$ is tangent to $E$ and so $\xi$ is an asymptotic direction.

Let $v$ be tangent to the discriminant curve at $(u_0,v_0)$. Then $D_v\xi=-f_*(Sv)$, and $Sv$ is not a multiple of $w$, since 
$$
Sv=(S-I)v+v
$$
and the first parcel of the second member is a multiple of $w$. We conclude that $D_v\xi$ does not belong to the tangent plane of $E$. But if $(u_0,v_0)$ were parabolic, $D_v\xi$ would be tangent to $E$, for any $v$, and so we conclude that $E$ is hyperbolic
at $(u_0,v_0)$. 
\end{proof}

\subsection{Subparabolic curve outside the discriminant}

To study the subparabolic curve, we shall consider the generic model $f(u,v)=(u,v,0)$ and $\xi(u,v)=(\xi_1,\xi_2,1)$. We write 
$$
\xi_u=-(a,c), \ \ \xi_v=-(b,d),
$$
with the restrictions $a_v=b_u$ and $c_v=d_u$. The eigenvalues are then
$$
\lambda_1=\tfrac{1}{2}\left( a+d+\sqrt{\delta}  \right), \ \ \lambda_2=\tfrac{1}{2}\left( a+d-\sqrt{\delta}  \right),
$$
where $\delta=(a-d)^2+4bc$ (\cite{Craizer-Garcia}). An eigenvector associated to $\lambda_2$ is 
$$
w=2\left(\lambda_2-d,c\right)=(a-d-\sqrt{\delta},2c).
$$

We shall now calculate the $1$-subparabolic curve. By Proposition \ref{Prop:Subparabolic}, the condition for a $1$-subparabolic point is $[f_w,f_{ww},\xi]=0$. In our model, this condition reduces to $[f_w,f_{ww}]=0$. We have that
$$
f_w=\left(a-d-\sqrt{\delta}, 2c \right).
$$
Then $f_{ww}$ is given by 
$$
\left(  \left( a_u-d_u-\frac{\delta_u}{2\sqrt{\delta}} \right) (a-d-\sqrt{\delta}) +  \left(a_v-d_v-\frac{\delta_v}{2\sqrt{\delta}} \right)2c, 2c_u(a-d-\sqrt{\delta})+4cc_v \right).
$$
Thus the condition $[f_w,f_{ww}]=0$ is given by
$$
c_u(a-d-\sqrt{\delta})^2+2cc_v(a-d-\sqrt{\delta})=c(a-d-\sqrt{\delta}) \left( a_u-d_u-\frac{\delta_u}{2\sqrt{\delta}} \right)+
$$
$$
+2c^2 \left(a_v-d_v-\frac{\delta_v}{2\sqrt{\delta}} \right).
$$
So the equation of the $1$-subparabolic set is
$$
c_u(a-d)^2+2cc_v(a-d)+c_u\delta-\sqrt{\delta}(2cc_v+2c_u(a-d)) =
$$
$$
c(a-d)(a_u-d_u)+2c^2(a_v-d_v)+\frac{c}{2}\delta_u-\sqrt{\delta}c(a_u-d_u)-c(a-d)\frac{\delta_u}{2\sqrt{\delta}}-c^2\frac{\delta_v}{\sqrt{\delta}}.
$$
Thus
$$
\sqrt{\delta} \left(  c_u(a-d)^2+2cc_v(a-d)+c_u\delta  +c(d-a)(a_u-d_u)+2c^2(d_v-a_v)-\frac{c}{2}\delta_u\right)
$$
$$
=\delta(2cc_v+2c_u(a-d)-c(a_u-d_u))-c(a-d)\frac{\delta_u}{2}-c^2\delta_v.
$$
or equivalently, using that $a_v=b_u$ and $c_v=d_u$, we obtain
$$
\sqrt{\delta} \left(  2c_u(a-d)^2+2c_ubc  +2c(d-a)(a_u-2d_u)+2c^2(d_v-2a_v)\right)
$$
$$
=(a-d)^2(2c_u(a-d)-2c(a_u-2d_u))
$$
$$
+2bc(3c_u(a-d)-2c(a_u-2d_u))-2c^2((a-d)(2b_u-d_v)+2b_vc).
$$

We have thus proved the following lemma:

\begin{lemma}
The equation of the $1$-subparabolic set is given by
$$
A\sqrt{\delta}=B,
$$
where
$$
A=c_u((a-d)^2+bc) +c(d-a)(a_u-2d_u)+c^2(d_v-2a_v),
$$
$$
B= -c((a - d)^2 + 2bc)(a_u - 2d_u) + ((a - d)^2 + 3bc)(a - d)c_u - c^2(a - d)(2a_v - d_v) - 2c^3b_v.
$$
\end{lemma}

\begin{Proposition}
Outside the discriminant curve, the $1$-subparabolic set is generically a smooth regular curve. Similarly, the same holds for the $2$-subparabolic set. 
\end{Proposition}

\begin{proof}
At a non-discriminant point, we may assume that $a=a_0,b=0,c=0,d=d_0$ at $(u,v)=(0,0)$, with $d_0-a_0>0$. 
Then $\lambda_1=d_0$ and $w(0,0)=2(a_0-d_0,0)$. Moreover, 
$$
\sqrt{\delta(0,0)}=d_0-a_0,\ \ A(0,0)=c_u(d_0-a_0)^2, \ \ B(0,0)=-c_u(d_0-a_0)^3.
$$
Thus the origin belongs to the $1$-subparabolic curve if and only if $c_u(0,0)=0$. Assuming this condition, 
$$
A_u=c_{uu}(d_0-a_0)^2, \ \ B_u=-c_{uu}(d_0-a_0)^3.
$$
Thus the derivative of $A\sqrt{\delta}-B$ with respect to $u$ at the origin is given by
$2c_{uu}(d_0-a_0)^3$. Since generically $c_{uu}\neq 0$, the proposition is proved. 
\end{proof}

By similar calculations, one obtain that the equation of the $2$-subparabolic curve is given by $A\sqrt{\delta}=-B$.
If we consider $P=0$, where 
\begin{equation}
P=A^2\delta-B^2,
\end{equation}
we obtain the equation of the union of the $1$-subparabolic curve with the $2$-subparabolic curve, that we shall call simply 
the subparabolic curve.

\subsection{Contact of the subparabolic and discriminant curves}\label{sec:ContactSubParabolic}

\begin{Proposition}
The subparabolic curve intersect the discriminant curve only at its singular points. Moreover, at singular discriminant points, the subparabolic and discriminant curves have generically a quadratic tangency.
\end{Proposition}

\begin{proof}
Assuming the origin is a point of the discriminant set, we may assume that $d(0,0)=a(0,0)=1$, $c(0,0)=1$, $b(0,0)=0$. Then the origin is singular if and only if $b_v(0,0)=0$. Observe that $\delta(0,0)=0$ and 
$$
B(0,0)=-2b_v(0,0),
$$
which says that the subparabolic curve intersect the discriminant curve only at singular points. Assume now that the origin
is a singular discriminant point. Then $\delta=\delta_v=B=0$ at the origin, which implies that $P_v(0,0)=0$. We have also that
$$
A(0,0)=d_v(0,0)-2b_u(0,0), \ \ \delta_u(0,0)=4b_u(0,0)
$$
are generically non-zero, which implies that $P_u(0,0)\neq 0$. Finally 
$$
P_{vv}(0,0)=4c^5b_{vv}\left(  (2b_u-d_v)d_v-2cb_{vv} \right)
$$
is also generically non-zero, which proves the proposition.
\end{proof}

\begin{example}
Consider 
$$
a=1+v,\ b=u+\frac{nv^2}{2},\ c=1,\ d=1+v+mv.
$$
Then 
$\delta=4u+(m^2+2n)v^2$,  $w(0,0)=\tfrac{\partial}{\partial v}$,
and so the origin is a singular discriminant point. Straightforward calculations show that
$A=m-1$, $B=(m-m^2-2n)v$ and
$$
P=4(m-1)^2u+2n(1-m^2-2n)v^2.
$$
Assume that $m>1$ and $m^2+2n-m>0$. Then the $1$-subparabolic curve is given by $P=0$, $v<0$, while the $2$-subparabolic curve is given by $P=0$, $v>0$. 
The ridge curve was calculated in \cite{Craizer-Garcia}. The $1$-ridge curve is given by $\bar{A}\sqrt{\delta}=\bar{B}$, while the $2$-ridge curve is given by $\bar{A}\sqrt{\delta}=-\bar{B}$, where
$$
\bar{A}=m+3,\ \ \bar{B}=(m-m^2-2n)v.
$$
In Figure \ref{fig:RidgeSubP}, we show the subparabolic and ridge curves close to the origin while in Figure \ref{fig:Focal} we show the corresponding curves at the focal set.

\begin{figure}[htb]
\centering
\includegraphics[width=0.3\linewidth]{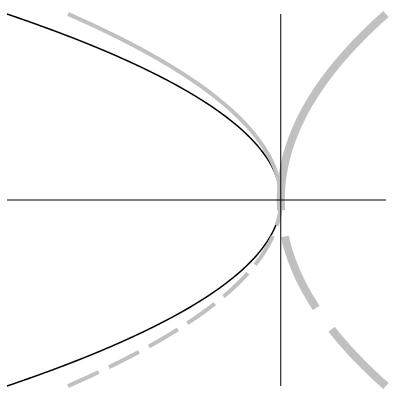}
 \caption{ Discriminant curve in solid black, subparabolic in thick gray ($1$ dashed, $2$ solid), ridge in thin gray ($1$ dashed, $2$ solid). We have used $m=3.2$ and $n=2.0$.}
\label{fig:RidgeSubP}
\end{figure}

\begin{figure}[htb]
\centering
\includegraphics[width=0.8\linewidth]{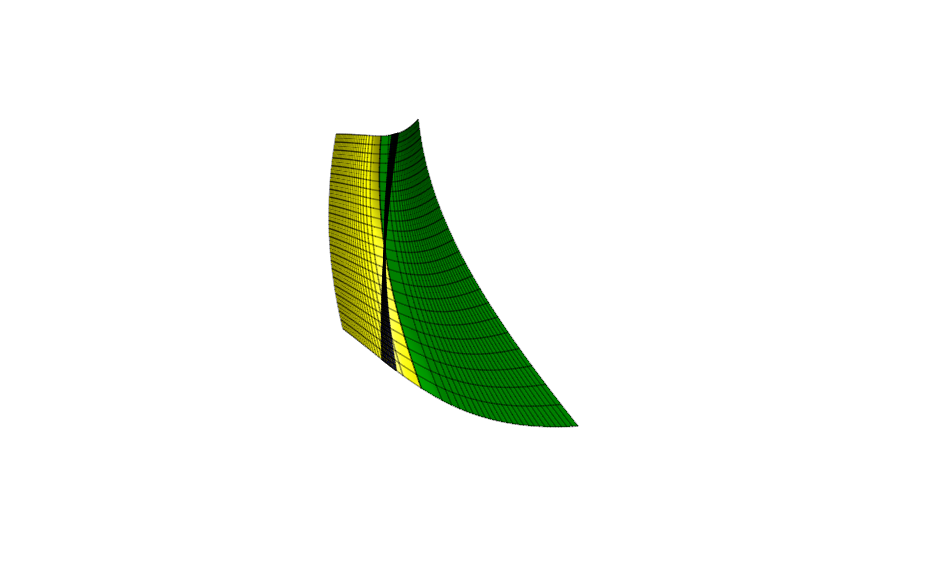}
 \caption{ Focal set at a singular discriminant curve. The dark region is the elliptic part of the focal set, bounded by the cuspidal edge (image of the ridges) and the parabolic curve (image of the subparabolics).}
\label{fig:Focal}
\end{figure}

\end{example}

\newpage

\section{Loewner's Type Conjectures}\label{sec:LoewnerAsymptotics}

The following conjecture can be found in (\cite{Gutierrez-Bringas},\cite{Gutierrez-Cidinha},\cite{Ballesteros}):

\paragraph{Conjecture 1}
Consider an isolated inflection point $(u_0,v_0)$ in the interior of the hyperbolic region of an immersion $G:U\to\R^4$. 
Then the index of the asymptotic line foliation of $G$ at $(u_0,v_0)$ is at most $1$.

\begin{remark} For generic immersions, it is well-known that this conjecture holds. In fact, it is proved in \cite{Garcia} that for a generic immersion $G:M^2\to\R^4$, the index of an inflection point $-1/2$ or $1/2$. 
\end{remark}

\paragraph{Conjecture 2} 
Consider an isolated umbilical point $(u_0,v_0)$ in the interior of the region of points with a pair of real 
principal directions of a line congruence $(f,\xi):U\to\mathbb{R}^3$. Then the index of the curvature lines of $(f,\xi)$ at $(u_0,v_0)$
is at most $1$. 

\begin{remark}It is proved in \cite{Craizer-Garcia-1} that Conjecture 2 holds if we assume that the line congruence
can be represented by an equiaffine semi-homogeneous pair $(f,\xi)$.
\end{remark}

\begin{Proposition}
Conjectures 1 and 2 are equivalent.
\end{Proposition}

\begin{proof}
If Conjecture 1 holds, then Conjecture 2 also holds by duality. Conversely, assume Conjecture 2 holds and let $G$ is an immersion.
By Section \ref{sec:Pedal}, $G$ is the dual of some equiaffine pair $(f,\xi)$, and then $(u_0,v_0)$ is an umbilical point
for this pair. Moreover, the curvature lines of $(f,\xi)$ correspond to asymptotic lines of $G$, which proves that Conjecture 1 also holds.
\end{proof}

\section{Conclusion}

Generically, a line congruence in $\mathbb{R}^3$ can be represented by an equiaffine pair $(f,\xi)$. The centroaffine duality sends an equiaffine lifting of $(f,\xi)$ to its projective pedal $G$, a surface in $\mathbb{R}^4$. This duality sends principal lines to asymptotic lines and consequently sends the discriminant curve of $(f,\xi)$ to the parabolic curve of $G$. It also takes the ridge curves of $(f,\xi)$ to the flat ridge curves of $G$ and the subparabolic curves of $S$ to certain curves on $G$ that we call flat subparabolic. We remark that if we change the equiaffine representation $(f,\xi)$ of the line congruence, the surface $G$ also changes, but its asymptotic lines, parabolic, flat ridge and flat subparabolic curves remain the same. 

The generic behavior of the principal lines and ridges at the discriminant curve were studied in \cite{Craizer-Garcia}. In this paper we describe the behavior of subparabolic curves at the discriminant curve of $(f,\xi)$, or equivalently, the flat subparabolic curve at the parabolic curve of the dual surface $G$. 

Loewner's conjecture says that, at an isolated umbilical point of a line congruence in $\mathbb{R}^3$, the index of the principal lines is at most $1$. A similar conjecture says that, at an isolated inflection point of a surface in $\mathbb{R}^4$, the index of the asymptotic lines is at most $1$. The above duality results imply that these two conjectures are in fact equivalent.

\bigskip\noindent
{\bf Data availability statement:} The present paper is theoretical and uses no experimental data.






\end{document}